\documentclass[10pt]{amsart}

\usepackage{amsmath}
\usepackage{amssymb}
\usepackage{amsthm}
\usepackage{graphics}
\usepackage{eucal}
\usepackage{bm}
\usepackage{mathtools}
\usepackage{mathrsfs}
\usepackage{color}

\usepackage[pdftex,bookmarks,colorlinks,breaklinks]{hyperref}
\usepackage{mathtools}
\newtheorem{theorem}{Theorem}[section]
\newtheorem{definition}{Definition}[section]
\newtheorem{proposition}{Proposition}[section]
\newtheorem{lemma}{Lemma}[section]

\newtheoremstyle{myremark}{10pt}{10pt}{}{}{\scshape}{.}{.5em}{}
\theoremstyle{remark}
\theoremstyle{myremark}
\newtheorem{remark}{Remark}
\usepackage{geometry}

\numberwithin{equation}{section}

\newcommand{\be}{\mathbf{e}}

\newcommand{\F}{\mathbb{F}}

\newcommand{\bq}{\mathbf{q}}
\newcommand{\bx}{\mathbf{x}}

\newcommand{\f}{\mathbf{f}}

\newcommand{\bz}{\mathbf{z}}

\newcommand{\by}{\mathbf{y}}

\newcommand{\Q}{\mathbb{Q}}
\newcommand{\K}{\mathcal{K}}

\newcommand{\FF}{\mathbb{F}_q((T^{-1}))}
\newcommand{\FZ}{\mathbb{F}_q[T]}

\newcommand{\Z}{\mathbb{Z}}
\newcommand{\R}{\mathbb{R}}
\newcommand{\N}{\mathbb{N}}

\newcommand{\sd}[1]{{\color{red} \sf SD: (#1)}}

\newcommand{\yx}[1]{{\color{blue} \sf YX: (#1)}}

\usepackage{mathtools}
\usepackage{tikz-cd}
\usepackage{graphicx}
\begin{document}

\title{SINGULAR VECTORS AND $\psi$-DIRICHLET NUMBERS over function field}
\author{Shreyasi Datta and  Yewei Xu}
\address{ Department of Mathematics, University of Michigan, Ann Arbor, MI 48109-1043}
\email{dattash@umich.edu} \email{ywhsu@umich.edu}

\date{}

\begin{abstract}
We show that the only $\psi$-Dirichlet numbers in a function field over a finite field are rational functions, unlike $\psi$-Dirichlet numbers in $\R$. We also prove that  there are uncountably many totally irrational singular vectors with large uniform exponent in quadratic surfaces over a positive characteristic field.
\end{abstract}

\maketitle
\section{Introduction}

	 \subsection{$\psi$-Dirichlet numbers}
	
	Following \cite{Kwa}, we define $\psi$-Dirichlet vectors in $\FF^n$ and we denote the set of those vectors as $D(\psi)$. For the definitions of norms in $\FF^n$, readers are referred to \S\ref{set}.
	\begin{definition}\label{psidi}
		Let $\psi \ : [t_0, +\infty) \to \mathbb{R_{+}}$ be a function. A vector $\bx=(x_1,\cdots,x_n)\in \FF^n$ is said to be $\psi$-Dirichlet if for all sufficiently large $Q>0$ there exists $\mathbf{0}\neq\bq\in \FZ^n$, $q_0\in\FZ$ satisfying the following system \begin{equation}\label{dirim}
		\begin{aligned}&\vert \bq\cdot\bx+q_0\vert<\psi(Q),\\&\Vert \bq\Vert\leq Q.\end{aligned}
		\end{equation}\end{definition}
		
	Let $\psi_c(Q)=\frac{c}{Q^n}$. If $\bx\in \FF^n$ is $\psi_c$-Dirichlet for every $c>0$, then $\bx$ is called singular vector.
	In recent years, $\psi$-Dirichlet vectors were studied in \cite{Kwa,Kwa2,KSY}. Even in the classical setting not much is known.\\
	\indent Diophantine approximation in function field has been a topic of interest since the work of Artin, \cite{Artin}, which developed the theory of continued fraction, and followed by Mahler's work in \cite{KURT}, which studied geometry of numbers in function field. For recent developments, we refer readers the survey \cite{Las}, and to \cite{Ghosh2007,Ganguly2022, K2003, KimNakada,KImlimPau, JasonYann, Boryann, ANF} for a necessarily incomplete set of references. There are many interesting similarities and contrasts between the theory of Diophantine approximation over the real numbers and in function field over finite fields. The main theorems in this paper show both of these features.
	
	\indent In \cite{Kwa2}, for a non-increasing function $\psi(t)<\frac{1}{t}$ with $t\to t\psi(t)$ non-decreasing, it was shown that $D(\psi)$ in $\R$ has zero-one law for Lebesgue measure depending on divergence or convergence of certain series involving $\psi$. Surprisingly, the same is not true over function field as we prove the following.
	
	\vspace{0.2 cm}
	\begin{theorem}\label{third}
	Let $\psi \ : [t_0, +\infty) \to \mathbb{R}_{+}$ be non-increasing.
	If $\psi(t) < \frac{1}{t}$ for sufficiently large $t$, then $D(\psi) = \mathbb{F}_q(T)$.
\end{theorem}
	The above theorem shows that analogue of the main theorem  in \cite{Kwa2} over function field becomes drastically different than the real case.   
	The main tool in proving Theorem \ref{third} is the use of continued fraction expansion.
	 \subsection{Plenty of singular vectors}
	 The second part of this paper deals with singular vectors in submanifolds of function fields. Note that, if $(x_1,\cdots ,x_n)\in \FF^n$ belongs to a rational affine hyperplane, then it must be singular. These are the most trivial singular vectors. In fact, the converse is also true when $n=1$ (ref.\cite{AG2}). So, we have the following definition to find vectors those can not be singular in a trivial manner.
	\begin{definition}
		We call a vector totally irrational vector if it is not inside a rational affine hyperplane of $\FF^n$.
	\end{definition}
	For $n>1$, in \cite{Kh} Khintchine showed the existence of infinitely many totally irrational singular vectors in $\R^n$. Moreover, Kleinbock, Moshchevitin and Weiss  in \cite{KNW} showed that for real analytic submanifolds (of dimension greater than $2$) which are not contained inside a rational affine subspace, there are uncountably many totally irrational singular vectors. In this paper, we prove analogous result for certain submanifolds in $\FF^n$.
	
\newpage

As a special case of our Theorem \ref{plenty}, we prove the following theorem.
	 \begin{theorem}\label{second}
Suppose $\text{char}(\FF) = p < \infty$. Let $U$ be an open subset of $\FF$.
We consider $S$ of the following two types:
\begin{itemize}
\item $S=\{(x,y, p_3(x,y),\cdots,p_n(x,y)) \ \mid \ x, y \in U\}\subset \FF^n$, where  and each $p_i(x,y)$ is a degree $2$ polynomial.
\item $S=\{(x,y,p(x,y))~|~ x,y\in U\}\subset \FF^3$, where $p(x,y)=\sum_{i=0}^m a_i^{p^i}x^{p^i}+\sum_{j=0}^n b_j^{p^j}y^{p^j}$.\end{itemize}   Suppose that $S$ is not contained inside any affine rational hyperplane,
then there exist uncountably many totally irrational singular vectors in $S$.
   
\end{theorem}
	 
	The main challenge comes from the lack of understanding about intersections of a surface and an affine subspace in the function field setting. Another difficulty comes due to total disconnectedness of function field. For real submanifolds, intersection of a connected analytic surface and an affine subspace is well understood due to \cite{EP}, \S 2. Both of these facts  were used in \cite{KNW} in a crucial manner. The proof in \cite{KNW} relies on understanding how `semianalytic' sets can spilt into connected analytic sets. This becomes difficult in $\F_q((T^{-1}))^n$, as the notion of semianalyticity is not well defined due to the lack of order and the space $\F_q((T^{-1}))$ is totally disconnected. That is why we had to tackle case by case and we prove the theorem for a class of submanifolds which is smaller than the class of submanifolds that was taken in \cite{KNW}.

\subsection{On uniform exponent}

One can define $\hat\omega(\cdot)$, as follows, which quantifies singularity of a vector.
\begin{equation}
     \hat\omega(\by):= \sup \left\{ \omega ~\left|~ \text { for all large enough } Q>0, \exists~ (q_0, \bq)\in \FZ^{n+1}\setminus\{0\} \text{ s.t.}  \begin{aligned}
     & \Vert \bq \cdot \by +q_0\Vert\leq \frac{1}{Q^\omega},\\
     & \Vert \bq\Vert\leq Q \end{aligned}
     \right.\right\}
\end{equation}
Dirichlet's Theorem (ref. \cite{GG}) gives that $\hat \omega(\by)\geq n$ for all $\by\in \FF^n$. Our next theorem verifies that for a certain analytic submanifolds in $\FF^n$, there are plenty of totally irrational vectors whose exponents $\hat\omega(\cdot)$ are infinity. 
   \begin{theorem}\label{plenty}
Let $\text{char}(\FF) = p < \infty$.
Let $U$ be an open subset of $\FF$.
We consider $S$ of the following two types:
\begin{itemize}
\item $S=\{(x,y, p_3(x,y),\cdots,p_n(x,y)) \ \mid \ x, y \in U\}\subset \FF^n$, where  and each $p_i(x,y)$ is a degree $2$ polynomial.
\item $S=\{(x,y,p(x,y))~|~ x,y\in U\}\subset \FF^3$ where $p(x,y)=\sum_{i=0}^m a_i^{p^i}x^{p^i}+\sum_{j=0}^n b_j^{p^j}y^{p^j}$.\end{itemize}   Suppose that $S$ is not contained inside any rational affine hyperplane,
then there exist uncountably many totally irrational $\by$ in $S$ such that $\hat\omega(\by)=\infty$.
   
\end{theorem}


\begin{remark}
	\leavevmode
	\begin{enumerate}
	 \item In Lemma \ref{surface}, we show that the above theorem is true for some higher dimensional submanifolds. 
		\item Theorem $2.4$ in \cite{AG2} shows that only \textit{Dirichlet improvable} numbers in function field are rational functions. Our Theorem \ref{third} generalizes the above mentioned result showing that even  $\psi$-Dirichlet numbers are also only rational functions. We note that the technique of \cite{AG2} is different than ours. 
	 
	\end{enumerate}
	
\end{remark}
\subsection{Norms and topology}\label{set}
In this section and in the following sections, we will use $\vert \cdot\vert$ (resp. $\Vert \cdot\Vert$ ) to denote norm in $\FF$ (resp. $\FF^n$), unless otherwise mentioned.
Let $p$ be a
prime and $q := p^r$, where $r \in \N$ and consider the finite field $\F_q$. We consider the integral domain $\F_q[T]$, the set of polynomials with coefficients in  $\F_q$. Then we consider the function field $\F_q(T )$. We define
a norm $\vert \cdot \vert$ on $\F_q(T )$ as follow:

$$ |0| := 0; \left\vert\frac{P}{Q}\right\vert
:= e^{\mathrm{degP} - \mathrm{degQ}}$$ for all nonzero $P,Q \in \F_q[T ]$ .
Clearly $| \cdot|$ is a nontrivial, non-archimedian and discrete absolute value in $\F_q(T )$.
The completion field of $\F_q(T )$ with respect to this absolute value is $\F_q((T^{-1}))$, i.e. the field of Laurent series over
$\F_q$.  We will denote the absolute value of $\F
_q((T^{-1}))$ by the same notation $|\cdot|$, is given as
follows. Let $a \in \F_q((T^{-1}))$,
$$\vert a\vert:=\left\{
\begin{aligned}
&0 \text{ if } a=0,\\
& e^{k_0} \text{ if } a=\sum_{k\leq k_0}a_kT^k, k_0 \in\Z, a_k \in \F_q \text{ and } a_{k_0} \neq 0. 
\end{aligned}\right.
$$
This clearly extends the absolute value $| \cdot |$ of $\F_q(T ) \to \F_q((T^{-1}))$ and moreover, the
extension remains non-archimedian and discrete. In the above, we call $k_0$ to be the degree of $a$, $\mathrm{deg}~ a$.  It is obvious that $\FZ$ is discrete in $\FF$. For any
$n \in \N$, throughout $\FF^n$ is assumed to be equipped with the supremum norm which
is defined as 
$\Vert\bx\Vert := \max_{1\leq i\leq n}|x_i|$ for all $\bx = (x_1, x_2, ..., x_n) \in \FF^n$ ,
and with the topology induced by this norm. Clearly $\FZ^n$ is discrete in $\FF^n$. Since
the topology on $\FF^n$ considered here is the usual product topology on $\FF^n$, it follows
that $\FF^n$ is locally compact as $\FF$ is locally compact. Note this construction $\FZ\subset\F_q(T)\subset \FF$ is similar to $\Z\subset\Q\subset\R$. Let $\lambda$ be the Haar measure on
$\FF^n$ which takes the value 1 on the closed unit ball $\Vert \bx\Vert = 1$.

\section{\texorpdfstring{$\psi$}{psi}-Dirichlet numbers in function field}\label{pp} 

\subsection{Continued fraction over function field} 
Suppose $a=\sum_{k\leq k_0} a_k T^k\in \FF$ where $a_{k_0}\neq 0$, we call $[a]:= \sum_{k\leq k_0}^0 a_k T^k$ as the integer part of $a$ and $\langle a\rangle =\sum_{k< 0}a_k T^k$ as the fractional part of $a$. Note that $\vert [a]\vert=e^{k_0}\geq 1$ if $a_{k_0}\neq 0$, and otherwise we have $[a]=0$. Also, note that $\vert \langle a\rangle \vert\leq 1$. This observation leads us to construct continued fraction expansion of $a$.  
An expression of the form $a_0 + \frac{1}{a_1 + \frac{1}{a_2 + \dots}}$ where $a_0, a_1,a_2 \in \mathbb{F}_q[T]$ is called a simple continued fraction; see \S $1$ in \cite{SchmidtWolfgang}. An expression of the form $\frac{p_n}{q_n} = a_0 + \frac{1}{a_1 + \frac{1}{a_2 + {\dots}_{+\frac{1}{a_n}}}}$, where $p_n,q_n,a_0,a_1,\dots,a_n \in \mathbb{F}_q[T]$ is called a finite continued fraction.
An element of $\mathbb{F}_q(T)$, can be represented as an unique finite continued fraction.
An $\alpha\in \FF\setminus \FZ$ can be represented as a simple continued fraction in the form of $[a_0,a_1,a_2,\dots]$ and 
we call the numbers $\frac{p_n}{q_n} = [a_0,a_1,\dots,a_n]$ the convergents of $\alpha$.
Note that $|q_n|$ is increasing as $n\to \infty$.  The relation between two consequitive covergents is given by the following equation: 
$$p_iq_{i+1} - p_{i+1}q_i = (-1)^{i+1} \text{ for } i \in \mathbb{Z}, i \geq -2.$$ Hence we have, $$\left \vert \frac{p_{n+1}}{q_{n+1}} - \frac{p_n}{q_n} \right\vert = \left\vert \frac{p_{n+1}q_{n} - p_{n}q_{n+1}}{q_{n}q_{n+1}} \right\vert = \left \vert \frac{\pm 1}{q_{n}q_{n+1}}\right\vert = \frac{1}{\vert q_{n} \vert \cdot \vert q_{n+1} \vert} \leq \frac{1}{\vert q_{n} \vert^2}.$$
In fact, by Equation $1.12$ in \cite{SchmidtWolfgang} we have \begin{equation}\label{cont1}\left\vert \alpha - \frac{p_n}{q_n} \right\vert = \left\vert \frac{1}{q_nq_{n+1}}\right\vert,
\end{equation} where  $\frac{p_n}{q_n}$ is a convergent of $\alpha \in \FF \setminus \mathbb{F}_q(T)$.
We recall the following definition and theorem of best approximation \cite{Malagoli}, \S 1.2.
\begin{definition}
We say a rational $\frac{a}{b}$ is the best approximation to some $\alpha \in \FF$ if for all $\frac{c}{d}$ such that $\vert d \vert \leq \vert b \vert$ we have $\vert b \alpha - a \vert \leq \vert d \alpha - c \vert$.
\end{definition}

\begin{theorem}
\label{cont3}
Let $\alpha \in \FF$ and let $(\frac{p_n}{q_n})_n$ be its convergents.
Let $p,q \in \mathbb{F}_q[T]$ with $q \neq 0$ be two relatively prime polynomials. 
Then $\frac{p}{q}$ is a best approximation to $\alpha$ if and only if it is a convergent to $\alpha$.
\end{theorem}

We want to recall the following Lemma $2.1$ from \cite{Kwa2} which was stated for real numbers. The verbatim proof will give the following lemma for function field. The proof uses the fact that convergents are best approximations, which we have by Theorem \ref{cont3}. In what follows $\frac{p_n}{q_n}$ are convergents of $x$.
\begin{lemma}\label{cont2}
Let $\psi \ : [t_0, +\infty) \to \mathbb{R}_{+}$ be non-increasing.
Then $x \in \FF \setminus\mathbb{F}_q(T)$ is $\psi$-Dirichlet if and only if $\vert \langle q_{n-1} x \rangle \vert < \psi(\vert q_n \vert)$ for sufficiently large $n$.
\end{lemma}

\subsection{Proof of Theorem \ref{third}} It is easy to see that $\F_q(T)\subset D(\psi)$. We want to show that $D(\psi)\subset \F_q(T)$. By Equation \eqref{cont1} for $x\in \FF\setminus \F_q(T)$ we have $ \vert \langle q_{n-1} x \rangle \vert = \frac{1}{\vert q_n \vert},$ $\forall ~n$. Since $\psi(t)<\frac{1}{t}$ for all large enough $t$, by Lemma \ref{cont2} we conclude that there is no $x\in \FF\setminus \F_q(T)$ such that $x$ is $\psi$-Dirichlet.

\section{Too many vectors with high uniform exponent}\label{too}

 	In this section we study totally irrational singular vectors in submanifolds of $\FF^n$. In dimension $n=1$, Theorem $2.4$ in \cite{AG2} implies that the set of numbers $y$ in $\FF$ that are singular is $\F_q(T)$. 
 	
\indent In order to state the main theorem of this section, we need to define the \textit{irrationality measure} function as follows. We follow the definition in \cite{KNW}. 
\begin{definition}
We define $\Phi: \FZ^n\setminus\{0\} \to \mathbb{R}_{+}$ to be a proper function if the set $\{\bq \in \FZ^n\setminus\{0\} \ : \Phi(\bq) \leq C\}$ is finite for any $C > 0$.
For any arbitrary $\Phi$ and any $\by \in \FF^n$, we define the irrationality measure function $\psi_{\Phi, \by}(t) := \min_{(q_0,\bq)\in\FZ\times \FZ^n\setminus\{\mathbf{0}\}, \Phi(\bq) \leq t}\vert \bq \cdot \by+ q_0 \vert$.
\end{definition}
We can now state one of the main theorems of this section.
\begin{theorem}\label{5t3}
Let $\text{char}(\FF) = p < \infty$. Let $U$ be an open subset of $\FF$. We consider $S$ of the following two types:
\begin{itemize}
\item $S=\{(x,y, p_3(x,y),\cdots,p_n(x,y)) \ \mid \ x, y \in U\}\subset \FF^n$, where each $p_i(x,y)$ is a degree $2$ polynomial.
\item $S=\{(x,y,p(x,y))~|~ x,y\in U\}\subset \FF^3$ where $p(x,y)=\sum_{i=0}^n a_i^{p^i}x^{p^i}+\sum_{j=0}^n b_j^{p^j}y^{p^j}$.\end{itemize} Suppose that $S$ is not contained inside any rational affine hyperplane.
Then for any proper function $\Phi: \FZ^n\setminus\{0\} \to \mathbb{R}_{+}$ and for any non-increasing function $\phi \ : \mathbb{R}_{+} \to \mathbb{R}_{+}$, there exist uncountably many totally irrationals $\by \in S$ such that $\psi_{\Phi,\by}(t) \leq \phi(t)$ for all large enough $t$.
\end{theorem}
As an application of the previous Theorem we get Theorem \ref{plenty}.
\begin{proof}[Proof of Theorem \ref{plenty}]
By taking $\Phi(\bq) = \Vert 
\bq \Vert$,  Theorem \ref{plenty} follows from Theorem \ref{5t3}.
\end{proof}

Let us recall Theorem $1.1$ from \cite{KNW}, which was proved for locally closed subsets of $\R^n$. The same proof verbatim will work for locally closed subsets in  $\FF^n$. It is noteworthy that the proof follows Khintchine's argument in \cite{Kh}. We define $\vert A\vert:=\max_{i=1}^{n+1} \vert a_i\vert$, where $A: a_1x_1+\cdots+a_nx_n=a_{n+1}$, and $(a_1,\cdots,a_{n+1})$ is a primitive vector in $\FZ^{n+1}$.

\begin{theorem}\label{c}
Let $S \subset \FF^{n}$ be a nonempty locally closed subset.
Let $\{L_1,L_2,\dots\}$ and $\{L_1',L_2',\dots\}$ be disjoint collections of distinct closed subsets of $S$, each of which is contained in a rational affine hyperplane in $\FF^n$, and for each $i$ let $A_i$ be a rational affine hyperplane containing $L_i$, assume the following hold:
\begin{enumerate}
\item[(a)] $$\bigcup_i L_i \cup \bigcup_j L_j' = \{x \in S \ : x \text{ is contained in a rational affine hyperplane}\};$$
\item[(b)] For each $i$ and each $\alpha> 0$, $$L_i = \overline{\bigcup_{|A_j| > \alpha} L_i \cap L_j};$$
\item[(c)] For each $i$, and for any finite subsets of indices $F$, $F'$ with $i \not \in F$, we have $$L_i = \overline{L_i - (\bigcup_{k \in F} L_k \cup \bigcup_{k' \in F'}L_{k'}')};$$
\item[(d)] $\bigcup_{i}L_i$ is dense in $S$.
\end{enumerate}\vspace{0.3 in}
\noindent 
Then for arbitrary $\Phi: \FZ^n\setminus\{0\} \to \mathbb{R}_{+}$  proper function and for any non-increasing function $\phi \ : \mathbb{R}_{+} \to \mathbb{R}_{+}$, there exist uncountably many totally irrationals $\by \in S$ such that $\psi_{\Phi,\by}(t) \leq \phi(t)$  for all large enough $t$.
\end{theorem}

We will call the property (a), (b), (c), and (d) defined above as ``property A". Let us recall the following Theorem $2.1.1$ in \cite{ZetaIgusa} which we are going to use throughout the rest of this section. 
\begin{theorem}\label{IFT}

Let $K$ be an arbitrary field and assume that for some $m$, $n$ every $F_i(x,y)$ in $F(x,y) = (F_1(x,y),  \dots, F_m(x,y))$ is in $K[[X,Y]] = K[[x_1,\dots,x_n, y_1,\dots, y_m]]$ satisfying $F_i(0,0) = 0$ and further $\frac{\partial (F_1,\dots, F_m)}{\partial (y_1,\dots, y_m)} \mid_{(0,0)} \neq 0$, in which $\frac{\partial (F_1,\dots, F_m)}{\partial (y_1,\dots, y_m)}$ is the Jacobian.
Then there exists a unique $f(x) = (f_1(x),\dots,f_m(x))$ with every $f_i(x)$ in $K[[x]] = K[[x_1,\dots, x_m]]$ satisfying $f_i(0) = 0$ and further $F(x,f(x)) = 0$.

\end{theorem}

\subsection{ When each $p_i(x,y)$ is a degree $2$ polynomial and $\FF$ is of any positive characteristic}

Let us consider $$S=\{(x,y,p_3(x,y),\cdots,p_n(x,y)) \ \mid \ x,y\in U\},$$ where $U$ is an open subset of $\FF$, and  $$p_i(x,y)=b_{i,1} x^2 + b_{i,2} xy + b_{i,3} y^2 + b_{i,4} x + b_{i,5} y + b_{i,6}$$ with $b_{1,i},b_{i,2},\dots,b_{i,6} \in \FF$ and $b_{i,1},b_{i,2},b_{i,3}$  not being zero simutaneously for $i=3,\cdots,n$. 
Let us take $A$ to be a rational affine hyperplane in $\FF^n$ and we assume that $S$ is not contained inside $A$. We can define 
$A$ by the linear equation $a_1 x_1 + a_2 x_2 + \cdots+a_n x_n = a_{n+1}$, where $(a_1,a_2,\cdots,a_{n+1})\in \mathbb{F}_q[T]^{n+1}$ is primitive. Note that $S\cap A$ is given by the solutions to the equation;
\begin{equation}
    f(x,y):= 0, 
\end{equation} where $$f(x,y)=\sum_{i=3}^n a_ip_i(x,y) + a_1 x + a_2 y - a_{n+1} .$$
We see that $f$ is a polynomial of degree  less than or equal to $2$. Now note that \begin{equation}\label{dx}\frac{\partial f}{\partial x} = \sum_{i=3}^n (2a_i b_{i,1} x + a_i b_{i,2} y) + (a_1 + \sum_{i=3}^n a_i b_{i,4})\end{equation} and 
\begin{equation}\label{dy}
\frac{\partial f}{\partial y} = \sum_{i=3}^n (a_i b_{i,2} x + 2 a_i b_{i,3} y) + (a_2 + \sum_{i=3}^n a_i b_{i,5}).\end{equation}
If $\frac{\partial f}{\partial x}(x_0,y_0) \neq 0$, then by Theorem \ref{IFT} we get a neighborhood of $(x_0,y_0)$, where  $y$ is a $\FF$-analytic function of $x$.
If $\frac{\partial f}{\partial y}(x_0,y_0) \neq 0$ then locally we can write $x$ as a $\FF$-analytic function of $y$.
Hence in order to find out all possible $(x_0,y_0)$ such that there is no neighborhood of $(x_0,y_0,p_3(x_0,y_0),\cdots,p_n(x_0,y_0)) \in S\cap A$ that is analytic curve in $S\cap A$,  we consider the linear system \begin{equation}\label{system}\begin{cases}  \frac{\partial f}{\partial x} = 0;\\ \frac{\partial f}{\partial y} = 0.\end{cases}\end{equation} The corresponding coefficient matrix $M \in \text{Mat}_{2 \times 2}(\FF)$ of the system is $$\begin{bmatrix}\sum_{i=3}^n 2a_ib_{i,1} & \sum_{i=3}^n a_ib_{i,2} \\ \sum_{i=3}^n a_ib_{i,2} & \sum_{i=3}^n 2a_ib_{i,3} \end{bmatrix}.$$ First note that if  $a_3,\cdots,a_n$ are zero, then $S\cap A$ is an analytic curve as the equation of $A$ would be $a_1x+a_2y=a_{n+1}$. Therefore one of $a_3,\cdots,a_n$ must be nonzero, and without loss of generality we assume that $a_{3}\neq 0$. Next let us denote, $$b_k= \sum_{i=3}^n  a_ib_{i,k}$$ for $k=1,\cdots,6$.
With this setting, we have the following lemma.

\begin{lemma}
\label{l44}
If $b_2^2\neq 4  b_1b_3$, then there are at most finitely many points of $A \cap S$ such that their neighbourhood is not an $\FF$-analytic curve in $\FF^n$.
\end{lemma}

\begin{proof}
Note that 
 $\det(M) = 4 {b_1}{b_3}-b_2^2$.
Hence by the hypothesis, we know that the system has only one solution. This completes the proof.

\end{proof}

From the proof above we know that the key is to solve the following equations; \begin{equation}\label{crucial1}\begin{aligned}
 f(x,y) & = \sum_{i=3}^n a_i p_i(x,y)+ a_1 x + a_2 y - a_{n+1}\\
 & = b_1 x^2+b_2xy+b_3y^2+(a_1+b_4)x+(a_2+b_5)y+(b_6-a_{n+1})=0 , \end{aligned}\end{equation}
 \begin{equation}\label{crucial2}\frac{\partial f}{\partial x}(x,y) = 2{b_1}x+{b_2}y+(a_1+{b_4})=0 ,\end{equation} \begin{equation}\label{crucial3} \frac{\partial f}{\partial y}(x,y) = b_2 x + 2 b_3 y + (a_2 + b_5) = 0.\end{equation}

\begin{lemma}
\label{l45}
If \begin{equation}\label{w0}b_2^2= 4 b_1b_3\end{equation} and $b_2 \neq 0$, then there are at most finitely many points of $A \cap S$ that its neighbourhood is not an $\FF$-analytic curve in $\FF^n$.
\end{lemma}

\begin{proof}
Suppose that there exists no point that satisfies System \eqref{system}, then conclusion of the lemma holds trivially. 

Now suppose that there exists a point $(x_0,y_0)$ that satisfies System \eqref{system}. Using  Equation \eqref{w0} we have, \begin{equation}\label{w1}
b_2(a_2 + b_5) \stackrel{\eqref{crucial3}}{=} -b_2^2 x_0 - 2 b_3(b_2 y_0) \stackrel{\eqref{crucial2}}{=} -b_2^2 x_0 + 2 b_3(2 b_1 x_0 + (a_1 + b_4)) \stackrel{\eqref{w0}}{=} 2b_3(a_1 + b_4).\end{equation}
For any $x,y \in U$, 
\begin{equation}\label{w4}
\begin{aligned}
b_2^2f(x,y) & \stackrel{\eqref{crucial1}}{=} b_2^2(b_1 x^2 + b_2 xy + b_3 y^2) + b_2^2((a_1+b_4)x+(a_2+b_5)y) + b_2^2(b_6 - a_{n+1}) \\
& \stackrel{\eqref{w0}}{=} b_3(2b_1 x + b_2 y)^2 + b_2^2((a_1+b_4)x+(a_2+b_5)y)  + b_2^2(b_6-a_{n+1})\\
& \stackrel{\eqref{w1}}{=} b_3(2b_1 x + b_2 y)^2 + b_2(a_1+b_4)(b_2x+2b_3y) + b_2^2(b_6-a_{n+1}).
\end{aligned}
\end{equation}

For any point that satisfies System \eqref{system},
\begin{equation}\label{fform}\begin{aligned}
&b_3(2b_1 x + b_2 y)^2 + b_2(a_1+b_4)(b_2x+2b_3y) + b_2^2(b_6-a_{n+1}) \\
& \stackrel{\eqref{crucial2}, \eqref{crucial3}}{=} b_3(a_1 + b_4)^2 - b_2(a_1 + b_4)(a_2 + b_5) + b_2^2(b_6 - a_{n+1}) \\
& \stackrel{\eqref{w1}}{=} b_2^2(b_6 - a_{n+1}) - b_3(a_1 + b_4)^2.
\end{aligned}
\end{equation}

Suppose that $b_2^2(b_6-a_{n+1}) \neq b_3(a_1+b_4)^2$. 
We know that Equation \eqref{crucial1} would never be satisfied for any points satisfying the System \eqref{system}. Therefore there is no point in $S\cap A$ such that its neighborhood is not an $\FF$-analytic curve.

Now suppose that \begin{equation}\label{w3}b_2^2(b_6-a_{n+1}) = b_3(a_1+b_4)^2.\end{equation}
Let $g(x) := -\frac{2b_1 x + (a_1 + b_4)}{b_2}$ and
$\gamma$ be $\{(x,g(x),p_3(x,g(x)),\cdots,p_n(x,g(x))) \ \mid \ x \in U\}$. Then clearly any point in $\gamma$ satisfies \eqref{crucial2}, \eqref{crucial3} and using \eqref{w3} one can see that any point in $\gamma$ also satisfies \eqref{crucial1}. Hence we have $\gamma \subseteq S \cap A$. Note here Equations \eqref{crucial2} and \eqref{crucial3} are essentially the same. 

Now for any point in $S\cap A$, we have $$ \begin{aligned}& f(x,y)=0\\ & ~ \stackrel{\eqref{w4}}{\implies} b_3(2b_1 x + b_2 y)^2 + b_2(a_1+b_4)(b_2x+2b_3y) + b_2^2(b_6-a_{n+1})=0\\
& \stackrel{\eqref{w3}, \eqref{w0}}{\implies} b_3 (2b_1 x+b_2y +a_1+b_4)^2=0\\
& ~\implies y= -\frac{2b_1 x + (a_1 + b_4)}{b_2}.
\end{aligned}$$ The last equality holds because $b_3\neq 0 $ as $b_2\neq 0$.
 Since $\gamma$ is already an analytic curve, this completes the proof of the lemma.  

\end{proof}

\begin{lemma}
\label{l46}
If $b_2^2= 4 b_1b_3$ and $b_2 = 0$, then there are at most finitely many points of $A \cap S$ that its neighbourhood is not an $\FF$-analytic curve in $\FF^n$.
\end{lemma}

\begin{proof}

If $b_1=0,b_2=0,b_3=0$ and there exists one point whose neighborhood is not an $\FF$-analytic curve, then $S\cap A$ being nonempty implies $S\cap A=A$, because $f(x,y)= b_6-a_{n+1}=0$. This contradicts the standing assumption that $S$ is not contained inside $A$.\\
If $\text{char}(\FF) \neq 2$, then $b_1 b_3 = 0$. Since $b_1$, $b_2$ and $b_3$ cannot be zero simultaneously, assume without loss of generality that $b_1 \neq 0$ and $b_3 = 0$. 

Suppose that there exists no point that satisfies System \eqref{system}, then conclusion of the lemma holds.
So let us assume that there exists a point $(x_0,y_0)$ that satisfies System \eqref{system}, then Equation \eqref{crucial3} gives us $a_2+b_5=0$.
By Equation \eqref{crucial1}, we have 
$$\begin{aligned} &b_1 x^2 + (a_1 + b_4) x + (b_6 - a_{n+1})=0.\end{aligned}$$
At most two $x$ can satisfy the above equation, 
say they are $x_1$ and $x_2$ respectively.
Then $$\gamma_1 = \{(x_1,y,p_3(x_1,y),\cdots,p_n(x_1,y)) \ \mid \ y \in U\}$$ and $$\gamma_2 = \{(x_2,y,p_3(x_2,y),\cdots,p_n(x_2,y)) \ \mid \ y \in U\}$$ are both analytic curves and they are inside $S\cap A$. 

In addition, since 
$$f(x,y) = b_1 x^2 + (a_1 + b_4)x + (b_6 - a_{n+1}),$$
we know that any point on $S \cap A$ must have an $x$-coordinate equal to $x_1$ or $x_2$, which means that it is in $\gamma_1$ or $\gamma_2$.
In other words, $S \cap A = \gamma_1 \sqcup \gamma_2$.
Thus the proof is complete for $\mathrm{char}(\FF)\neq 2$.

If $\text{char}(\FF) = 2$, then $\frac{\partial f}{\partial x} = a_1 +  b_4$ and $\frac{\partial f}{\partial y} = a_2 + b_5$.
If either of the two is nonzero then the conclusion of the lemma holds. Otherwise, for every $x,y$  we have $$f(x,y) = b_1 x^2 + b_3 y^2 + (b_6 - a_{n+1}).$$ 
Since $b_1$, $b_2$ and $b_3$ cannot be zero simultaneously, assume without loss of generality that $b_3 \neq 0$. 
Suppose $\frac{b_1}{b_3}$ is not square. Now if we have two points $(x_0,y_0,p_3(x_0,y_0),\cdots, p_n(x_0,y_0))$  and $(x_1,y_1,p_3(x_1,y_1),\cdots,p_n(x_1,y_1))$ in $S\cap A$, then $$\begin{aligned}
& \ \ \ \ \ \ \ b_1 x_0^2+b_3y_0^2= b_1x_1^2+b_3y_1^2\\
& \implies b_1(x_1-x_0)^2=b_3(y_1-y_0)^2.\end{aligned}$$ The above gives a contradiction to the assumption that $\frac{b_1}{b_3}$ is not a square. Hence in this case there could be at max one point in $S\cap A$. \\
\indent Now let us assume $\frac{b_1}{b_3}=\alpha^2$ for some $\alpha\in \FF$.
Any point in $S \cap A$ satisfies the equation $b_1 x^2 + b_3 y^2 = (a_{n+1} - b_6)$.
This is equivalent to $\alpha^2 x^2 + y^2 = \beta$, where $\beta = \frac{a_{n+1} - b_6}{b_3}$.
Suppose that $(x_0,y_0,p_3(x_0,y_0),\cdots, p_n(x_0,y_0))$ is a point in $S \cap A$, which implies $\alpha^2 x_0^2 + y_0^2 = \beta$.
Suppose that $$(x(t),y(t),p_3(x(t),y(t)),\cdots, p_n(x(t),y(t)))$$ is in $S \cap A$.
For any $t \in \FF$, let $x(t) = x_0 + t$.
Then the corresponding $y(t)$ can be given as $(y(t))^2 = \beta - \alpha^2 (x(t))^2 = (\beta - \alpha^2 x_0^2) - \alpha^2  t^2 = y_0^2 + \alpha^2  t^2$.
Therefore $y(t) = y_0 + \alpha t$, and this shows that $S \cap A$ gives an analytic curve.

\end{proof}
Combining Lemma \ref{l44}, Lemma \ref{l45} and Lemma \ref{l46} we get the following proposition.
\begin{proposition}\label{prop41}
Let $S=\{(x,y,p_3(x,y),\cdots, p_n(x,y))~|~x,y\in U\}$, where $U$ is an open set of $\FF$, and each $p_i(x,y)$ is a degree $2$ polynomial and $A$ be an affine rational hyperplane in $\FF^n$. Suppose that $S$ is not contained inside $A$ then there are at most finitely many points of $ S\cap A$ such that its neighbourhood is not an $\FF$-analytic curve in $\FF^n$.
\end{proposition}
 By the above proposition we have that $S\cap A\setminus J= \cup_{\bz\in S\cap A} \gamma(\bz)$, where $J$ is a finite set of points which do not have an $\FF$-analytic curve as neighborhood, and $\gamma(\bz)$ is an $\FF$-analytic curve which is a neighborhood of $\bz$ in $S\cap A$. Also note that $\gamma(\bz)$ is open and closed. Since $S\cap A\setminus J$ is a second countable space, we know that there exists a countable subcovering $\gamma_j$, i.e. $S\cap A\setminus J= \cup_{i} \gamma_i$.
 
\begin{theorem}\label{5td2}
Let $S$ be as in the previous proposition. 
There exist $\{L_i\}$, $\{L_j'\}$, $\{A_j\}$ as mentioned in Theorem \ref{c}, that satisfy property A.
\end{theorem}

\begin{proof}
Let $\{A_i\}$ be the set of affine rational hyperplanes normal to one of $x$-axis or $y$-axis in $\FF^n$.
By possibly replacing $S$ with a smaller restriction on $x$ and $y$, we can ensure that for any $\zeta \in S$, $T_{\zeta}S$ is normal to $x$ and $y$-axis.
Now let us define $L_i = S \cap A_i$,
which are closed subsets and curves of $S$.

Next we define $\{L_j'\}$.
For any affine rational hyperplane $A$ that has a nonempty intersection with $S$, 
by proposition \ref{prop41}, we have that  $S\cap A$ is union of $\gamma_j$, excluding finitely many points.
Let $\{L_j'\} = \{\gamma_j : \forall i, \gamma_j \not \subset L_i\}$. Now we want to verify that these collections satisfy four hypotheses of Theorem \ref{c}.

Property (a) of Theorem \ref{c} follows directly from how we defined these sets.\\
\indent First let us consider those $L_i=S\cap A_i$, where $A_i: x=a$, $a\in \F_q(T)$. Now let us consider $A_{j_k}:y= \frac{b}{T^k}$, where $b\in \FZ$. Since $L_i=\{(a,y,p_3(a,y),\cdots,p_n(a,y)|y\in U\}$, and $\{\frac{b}{T^k}|k\geq m\}$ for some $m\in \N$, is dense in $U$, property (b) follows. \\
\indent Let $F$ and $F'$ are as in hypothesis (c) of Theorem \ref{c}. Let $$L_i=\{(a,y,p_3(a,y)\cdots,p_n(a,y))~|~x,y\in U\},$$ where $a\in \F_q(T)$. Note that each $L_i \cap L_k$ with $k \in F$ is either empty or consists of only a single point. 
For any $k' \in F'$, by the equation given above, $L_i \cap L_{k'}'=L_i\cap \gamma_{k'}$. We can write $\gamma_{k'}$ is a subset of $ a_1x+a_2y+\sum_{i=3}^n a_ip_i(x,y)-a_{n+1}=0$, and therefore, $L_i\cap\gamma_{k'}$ is a subset of $a_1a+a_2y+\sum_{i=3}^n a_ip_i(a,y)-a_{n+1}=0$. Hence only finitely many solution is possible and $L_i\cap L'_{k'}$ has no interior. The same proof will work when $L_i=A_i\cap S$ where $A_i$ is normal to $y$-axis. \\
\indent To verify property (d) of Theorem \ref{c} it is enough to observe that $S\cap A_i$ looks like $$(x,a,p_3(x,y),\cdots,p_n(x,y)~|~ x,y\in U\}$$ or $$(b,y,p_3(x,y),\cdots,p_n(x,y)~|~ x,y\in U\},$$ where $a,b\in \F_q(T)$.
Clearly they form a dense set in $S$.
\end{proof}

\subsection{A special case in higher degree $p(x,y)$}
\begin{proposition}
\label{r45}
Let $\text{char}(\FF) = p < +\infty$. Let $S = \{(x,y,p(x,y))~|~x,y\in U\}$, where $p(x,y) = \sum_{i=0}^{m} b_{i}^{p^i} x^{p^i} + \sum_{j=0}^{n} c_{j}^{p^j} y^{p^j}$ and $U$ is an open subset in $\FF$. 
There are at most finitely many points of $S \cap A$ such that its neighbourhood is not an $\FF$-analytic curve in $\FF^3$.
\end{proposition}

\begin{proof} 
Without loss of generality let us assume that $S \cap A$ is nonempty, and that there exists at least one point of $S \cap A$ whose neighbourhood is not an $\FF$-analytic curve in $\FF^3$. This implies $a_3\neq 0$, and we can also assume $a_3=1$ after normalization.  This implies that $a_1+b_{0}=a_2+ c_{0}=0$.
Suppose also without loss of generality that $m \geq n$, $b_m \neq 0$, and $c_n \neq 0$.
The intersection $S\cap A$ is given by $
f(x,y) = 0,
$ where 
$$\begin{aligned}f(x,y) &= \sum_{i=1}^{m} b_i^{p^i} x^{p^i} + \sum_{j=1}^{n} c_j^{p^j} y^{p^j}-a_{4}\\
& =\left (\sum_{i=1}^m b_i^{p^i} x^{p^{i-1}} + \sum_{j=1}^n c_j^{p^j} y^{p^{j-1}}\right)^p-a_4\end{aligned}$$ Since $S\cap A$ is nonempty, $a_4$ must have a $p$-th root, say it is $a_{4,0}^p=a_4$. Therefore, we have that $S\cap A$ is given by $$f_0(x,y):=\sum_{i=1}^m b_i^{p^i} x^{p^{i-1}} + \sum_{j=1}^n c_j^{p^j} y^{p^{j-1}}-a_{4,0}.$$
If $S\cap A$ has one point which has no neighbourhood that is an $\FF$-analytic curve, we have $b_1=c_1=0$. Hence we can write  $f_0(x,y)=0$ as, $$
f_0(x,y) = \left(\sum_{i=2}^{m} b_{i}^{p^i} x^{p^{i-2}} + \sum_{j=2}^{n} c_{j}^{p^j} y^{p^{j-2}}\right)^p-a_{4,0}.$$ Since $S\cap A$ is nonempty $a_{4,0}$ must have a $p$-th root, say it is $a_{4,1}^p=a_{4,0}$. Since $f_0(x,y)=(f_1(x,y))^p$, where $$f_1(x,y):=\sum_{i=2}^{m} b_{i}^{p^i} x^{p^{i-2}} + \sum_{j=2}^{n} c_{j}^{p^j} y^{p^{j-2}}-a_{4,1},$$ we have that $S\cap A$ is defined by $f_1(x,y)=0$. 

We use an induction to derive the desired results.
For any $k \in \N$ satisfying $1 \leq k \leq n-1$, $ b_{1} = \dots = b_{k-1} = 0$, and $c_{1} = \dots = c_{k-1} = 0$, assume that $S \cap A$ is given by $f_k(x,y) = 0$, where $$
\begin{aligned}& f_k(x,y) = \sum_{i=k+1}^{m} b_{i}^{p^i} x^{p^{i-k-1}} + \sum_{j=k+1}^{n} c_{j}^{p^j} y^{p^{j-k-1}}-a_{4,k}
\end{aligned}$$
for some $a_{4,k} \in \FF$.
We have $\frac{\partial f_k}{\partial x} = b_{k+1}$ and $\frac{\partial f_k}{\partial y} = c_{k+1}$.
By System \eqref{system}, we derive that $b_{k+1} = c_{k+1} = 0$.
Now since
$$
\begin{aligned} f_k(x,y)
& = \sum_{i=k+2}^{m} b_{i}^{p^{i}} x^{p^{i-k-1}} + \sum_{j=k+2}^{n} c_{j}^{p^j} y^{p^{j-k-1}}-a_{4,k} \\
& = \left(\sum_{i=k+2}^{m} b_{i}^{p^i} x^{p^{i-k-2}} + \sum_{j=k+2}^{n} c_{j}^{p^j} y^{p^{j-k-2}}\right)^p - a_{4,k}, 
\end{aligned}$$
and by our assumption that $S \cap A$ is nonempty, we know that $a_{4,k}$ must have a $p$-th root, say it is $a_{4,(k+1)} \in \FF$, i.e., $a_{4,(k+1)}^p = a_{4,k}$.
Define $$f_{k+1}(x,y) := \sum_{i=k+2}^{m} b_{i}^{p^i} x^{p^{i-k-2}} + \sum_{j=k+2}^{n} c_{j}^{p^j} y^{p^{j-k-2}} - a_{4,(k+1)}.$$
Then 
$$
f_k(x,y)=
 (f_{k+1}(x,y))^p.
$$
This implies that the intersection $S \cap A$ is given by the equation $f_{k+1}(x,y) = 0$.

Repeating the steps above we see that the intersection $S \cap A$ is given by $f_{n-2}(x,y) = 0$.
Also from the induction above, we see that 
$$\begin{aligned}
f_{n-2}(x,y) & = c^{p^n}_{n} y - a_{4,n-2} + \sum_{i=n}^{m} b_{i}^{p^i} x^{p^{i-n+1}}.
\end{aligned}$$
This tells us that $S \cap A$ is completely given by the curve $$\{(x,g(x),p(x,g(x))) \ \mid \ x\in U\},$$ where $$g(x) := \frac{1}{c_n^{p^n}}\left(a_{4,n-2} - \sum_{i=n}^{m} b_{i}^{p^i} x^{p^{i-n+1}}\right), \text{ which is analytic.}$$

\end{proof}
The exact same proof as Theorem \ref{5td2} with suitable changes  gives the following theorem.
\begin{theorem}\label{5tfin}
Let $S = \{(x,y,\sum_{i=0}^m b_i^{p^i} x^{p^i} + \sum_{j=0}^n c_j^{p^j} y^{p^j})\}$, where $x,y,b_i,c_j \in \FF$, and not all $b_i$s or $c_j$s are being zero. 
Then there exists $\{L_i\}$, $\{L_j'\}$, $\{A_j\}$ that satisfies property A.
\end{theorem}

We now have everything we need to prove the main theorem of this section:

\begin{proof}[Proof of Theorem \ref{5t3}]
 Theorem \ref{5td2} and Theorem \ref{5tfin} guarantee that the hypotheses of Theorem \ref{c} are met by surfaces considered in Theorem \ref{5t3}. Therefore by Theorem \ref{c}, Theorem \ref{5t3} follows. 
\end{proof}

\subsection{Surface to higher dimensional manifold}

The following theorem is somewhat an analogue to Lemma $3.5$ in \cite{KNW}. 

\begin{theorem}\label{surface}
Suppose that for $k \geq 3$ and $B(0,1)$ is the open and closed  ball in $\FF$ of radius $1$. Let $M$ be a $k$-dimensional submanifold of $\FF^n$ which is the image of $B(0,1)^k$ under an immersion $\f \ : B(0,1)^k\to \FF^n$.
Suppose that $(\f,\lambda_k)$ is nonplanar, where $\lambda_k$ is the Lebesgue measure in  $B(0,1)^k$.
Then there exists $\by \in B(0,1)^{d-2}$ such that the surface $ M_{\by}:=\f_{\by}((B(0,1)^2)$ is not contained inside an rational affine hyperplane, where $\f_{\by} \ : B(0,1)^{2}\to \FF^n$, $\f_\by(x_1,x_2) = \f(x_1,x_2,\by)$.
\end{theorem}
\begin{proof}
We will prove by contradiction. Suppose for every $\by\in B(0,1)^{k-2}$ there exists a rational affine subspace $A_\by$ such that $B(0,1)^2\times \{\by\}\subset \f^{-1} (A_\by\cap M) \implies B(0,1)^k= \cup_{A\text{ is a rational affine subspace}}  \f^{-1}(A\cap M).$ By Baire category theorem there is one $A$ such that $\f^{-1}(A\cap M)$ contains an open ball inside $B(0,1)^k$. This contradicts the fact that $(\f,\lambda_k)$ is nonplanar.  
\end{proof}
\begin{remark}
 \leavevmode
\begin{enumerate}
\item In the above theorem we don't need to consider the manifold to be analytic but we need a stronger assumption that $(\f,\lambda)$ is nonplanar as compared to the manifold being not inside an affine hyperplane.
\item The above theorem shows it is enough to prove Theorem \ref{plenty} for surfaces. 

\end{enumerate}
\end{remark}

\subsection{ Product of perfect sets}
Let us recall that a subset of $\FF$ is called perfect if it is compact and has no isolated points. If $L=\{\by\in \FF^n~|~ \sum_{i=1}^n a_iy_i=a_0\}$, we can define $\vert L\vert:=\max_{i=1}^n \vert a_i\vert$. The proof of the next proposition will be exactly the same as the proof of Theorem 1.6 in \cite{KNW}.
\begin{proposition}\label{perf}
Let $n \geq 2$ and let $S_1,\dots,S_n$ be perfect subsets of $\FF$ such that $\mathbb{F}_q(T) \cap S_1$ is dense in $S_1$ and $\mathbb{F}_q(T) \cap S_2$ is dense in $S_2$.
Let $S = \prod_{j=1}^{n} S_j$.
Then there exist a collection of $\{L_i\}$, $\{L_j'\}$, $\{A_i\}$ of $S$ that satisfy property A.
\end{proposition}
Thus we have the following theorem combining Theorem \ref{c} and Proposition \ref{perf}.
\begin{theorem}\label{perfset}
Let $n \geq 2$ and let $S_1,\dots,S_n$ be perfect subsets of $\FF$ such that $(\mathbb{F}_q(T) \cap S_1)$ is dense in $S_1$ and $(\mathbb{F}_q(T) \cap S_2)$ is dense in $S_2$.
Let $S = \prod_{j=1}^{n} S_j$.
Then there exist uncountably many totally irrational singular vectors in $S$.
\end{theorem}

\subsection*{Acknowledgements} We thank Ralf Spatzier and Anish Ghosh for many helpful discussions and several helpful remarks which have improved the paper. We also thank Subhajit Jana for several helpful suggestions about the presentation of this paper. The second named author thanks University of Michigan for providing help as this project started as a Research Experience for Undergraduates project in summer 2021.

\bibliographystyle{abbrv}
\bibliography{bib}

\end{document}